\documentclass[11pt,a4paper]{article}
\usepackage{amsmath}
\usepackage{amsthm}
\usepackage{amssymb}
\usepackage{url}
\usepackage{amsfonts}
\usepackage{setspace}  
\usepackage{anysize}

\usepackage{pstricks,pst-node,pst-text,pst-3d}
\usepackage{amscd}        
\usepackage{graphics}
\usepackage{color}
\usepackage[all]{xy}
\usepackage{amsthm}   
\usepackage{natbib}
\usepackage{mathrsfs}
\usepackage[all]{xy}
\usepackage{textcomp}

\numberwithin{equation}{section} \theoremstyle{plain}
\newtheorem{theorem}{Theorem}[section]
\newtheorem{lemma}[theorem]{Lemma}
\newtheorem{corollary}[theorem]{Corollary}
\newtheorem{proposition}[theorem]{Proposition}
\theoremstyle{definition}
\newtheorem{definition}[theorem]{Definition}

\theoremstyle{remark}

\def\os{{\cal O}}

\def\qco{\mathfrak{Qcoh}}

\usepackage{mathrsfs}

\newcommand{\Hom}{\operatorname{Hom}}
\newcommand{\Filt}{\operatorname{Filt}}
\newcommand{\Ext}{\operatorname{Ext}}
\newcommand{\id}{\operatorname{id}}

\newcommand{\Ker}{\operatorname{Ker}}

\newcommand{\res}{\operatorname{res}}
\newcommand{\Ann}{\operatorname{Ann}}
\newcommand{\Sum}{\operatorname{\mathbf{Sum}}}

\renewcommand{\theenumi}{\roman{enumi}}
\renewcommand{\labelenumi}{(\theenumi)}
\begin{document}
\pagestyle{myheadings}

\enlargethispage{\baselineskip}
\title{Locally torsion-free quasi--coherent sheaves
}

\date{}
\author{  Sinem Odaba\c{s}{\i}\footnote{Departamento de Matem\'{a}ticas, Facultad de Matem\'{a}ticas, Universidad de Murcia, Murcia, Spain.
e-mail: \texttt{sinem.odabasi1@um.es}. } }

 \maketitle
\renewcommand{\theenumi}{\arabic{enumi}}
\renewcommand{\labelenumi}{\emph{(\theenumi)}}

\begin{abstract}
Let $X$ be an arbitrary scheme. The category $\qco(X)$ of quasi--coherent sheaves on $X$ is known that admits arbitrary direct products. However their structure seems to be rather mysterious. In the present paper we will describe the structure of the product object of a family of locally torsion-free objects in $\qco(X)$, for $X$ an integral scheme. Several applications are provided. For instance it is shown that the class of flat quasi--coherent sheaves on a Dedekind scheme $X$ is closed under arbitrary direct products, and that the class of all locally torsion-free quasi--coherent sheaves induces a hereditary torsion theory on $\qco(X)$. Finally torsion-free covers are shown to exist in $\qco(X)$.

\vspace{0.5cm}

\end{abstract}

\section{Introduction}
The  class of flat quasi--coherent sheaves on a scheme
$X$ has been extensively used during the last years, as a natural choice
for studying both the homotopy category and the derived category of
quasi--coherent sheaves (\cite{EGPT,Gillespie,Hovey,MurSal,Murfet,HosSal}).

On the other hand Gabber showed (see \cite[Lemma 2.1.7]{conrad} for a reference and \cite{relativehom} for a proof)
that the
category of quasi--coherent sheaves on an arbitrary scheme admits a family of
generators in the sense of \cite{Grothendieck}. Therefore this category has enough
injectives and direct products always exist on it. However it seems to be a hard task
to
know an explicit description of this object. This is partially so because, at
the level of sections, the direct product of modules is not well-behaved in
general with respect to localizations, or more generally, when tensoring by an
arbitrary module with respect to a commutative ring (direct products do not
commute with tensoring in general). But even in case that the tensor product
does commute with products with respect to finitely presented modules
(for instance when the ring $R$ is coherent) it is not clear whether the product
object in $\qco(X)$ can be computed from the product module of sections at each
affine open if we do not impose extra assumptions on the sheaf  of rings
${\mathcal O}_X$ attained to $X$ (for instance if ${\mathcal O}_X(U)$ is
finitely presented as ${\mathcal O}_X(V)$-module, for each affine open subsets
$U\subseteq V$).

The lack of an explicit description of the product object leads to
new and relevant questions on the class ${\rm Flat}(X)$ of {\it flat} quasi-coherent sheaves on $X$. For instance, Murfet in \cite[Remark B.7]{Murfet}
raises the question of whether ${\rm Flat}(X)$ is closed under products, for $X$ a noetherian scheme.
This property is crucial to showing that in ${\rm Ch}(A)$, the category of unbounded chain complexes of $A$-modules ($A$ commutative noetherian ring), the complex $Hom_A(I,I')$ is a complex of flat modules, for injectives $I,I'\in {\rm Ch}(A)$. We point out that the usual notion of flatness in $\qco(X)$ is not categorical, as it shown in \cite{ES}. Recently Saor\'{\i}n and \v{S}\v{t}ov\'{\i}\v{c}ek \cite[4.2]{Saorin} have given a positive answer to this question for Dedekind schemes. In their argument they use Crawley-Boevey's
characterization of preenveloping subcategories of the category of finitely presented objects in a locally finitely presented additive category with products (see \cite[Theorem 4.2]{cb}. So then they show that if $X$ is Dedekind, the category of finitely presented objects in $\qco(X)$ (the vector bundles) is preenveloping, obtaining as a byproduct that  its closure under direct limits, the class ${\rm Flat}(X)$, is closed under products.

If $X$ is affine, there is a canonical equivalence between ${\rm Flat}(X)$
and the class ${\rm Flat}(R)$ of flat $R$-modules, where $X=Spec(R)$. Now if $X$
is also Dedekind it is well known that ${\rm Flat}(R)$ coincides with the class
of torsion-free $R$-modules. So, for an arbitrary scheme, it makes sense to
define the class $\mathscr{F}$ of {\it locally torsion-free }
quasi--coherent sheaves as the class of $\mathcal F\in \qco(X)$ such that $\mathcal{F}(U)$ is a torsion-free $\mathcal{O}_X(U)$-module, for each affine open set $U$. This class contains ${\rm Flat}(X)$ in general, and indeed
it coincides with it for Dedekind schemes. Thus this paper is devoted to study the class $\mathscr{F}$. More precisely, in the first part
we will characterize the product object of a family of quasi--coherent sheaves
in $\mathscr{F}$ obtaining, as a consequence, the forementioned result of
Saor\'{\i}n and \v{S}\v{t}ov\'{\i}\v{c}ek for Dedekind schemes. The main result of this section is the following (Theorem \ref{prodtf}).

\medskip\par\noindent
{\bf Theorem I.} Let $X$ be an integral scheme. The direct product $\mathcal{F}$ of a family $\{\mathcal{F}_i\}_{i\in I}$ of torsion-free quasi-coherent sheaves in $\qco(X)$ is the largest quasi-coherent subsheaf of $\prod_{i\in I} \mathcal{F}_i$. More concretely, it is  of the form
$$\mathcal{F}=\sum_{\substack{\mathcal{M}\in \qco(X)\\\mathcal{M} \subseteq \prod_{i\in I} \mathcal{F}_i}} \mathcal{M}.$$
\bigskip\par

One of the consequences of this theorem is that for an integral scheme the class
$\mathscr{F}$ induces a hereditary torsion theory in the sense of \cite{dickson} in $\qco(X)$ (see also \cite{blas} for an extensive study about torsion theories in $\mathfrak{Sh}(X)$ (the category of sheaves of $\mathcal{O}_X$-modules) and $\qco(X)$).

Flat covers are shown that exist in \cite[Theorem 4.1]{relativehom}. In the second part of the paper (Section \ref{tfcover}) we show the existence of covers with respect to the class $\mathscr{F}$. This was known from the sixties in case $X$ is integral and affine. The result is due to \cite{torsionfree} (see also \cite{teply} for a more general version for arbitrary torsion theories). Thus the main theorem of this section (Theorem \ref{tfcover}) states:

\medskip\par\noindent
{\bf Theorem II.} Each quasi--coherent sheaf on an integral scheme has a locally torsion-free cover.

\section{Preliminaries}
\begin{definition}
Let $\mathcal{C}$ be a Grothendieck category. A  direct system of objects of $\mathcal{C}$, $(M_\alpha \mid \alpha \leq \lambda)$, is said to be a continuous system of monomorphisms if $M_0=0$, $M_\beta= \varinjlim _{\alpha< \beta} M_{\alpha}$ for each limit ordinal $\beta \leq \lambda$ and all the morphisms in the system are monomorphisms.

Let $\mathcal{S}$ be a class of objects which is closed under isomorphisms. An object $M$ of $\mathcal{C}$ is said to be $\mathcal{S}$-filtered if there is a continuous system $(M_\alpha \mid \alpha \leq \lambda)$ of subobjects of $M$ which is $M=M_{\lambda}$ and $M_{\alpha+1}/M_{\alpha}$ is isomorphic  to an object of $\mathcal{S}$ for each $\alpha< \lambda$.
\end{definition}

The class of $\mathcal{S}$-filtered objects in $\mathcal{C}$ is denoted by $\Filt(\mathcal{S})$. The relation $\mathcal{S} \subseteq \Filt(\mathcal{S})$ always holds. In the case of being $\Filt(\mathcal{S}) \subseteq \mathcal{S}$, the class $\mathcal{S}$ is said to be closed under $\mathcal{S}$-filtrations.

\begin{definition}
Let $\mathscr{F}$ be a class of objects of $\mathcal{C}$. A morphism $\phi : F \rightarrow M$ of $\mathcal{C}$ is said to be an $\mathscr{F}$-precover of $M$ if $F \in \mathscr{F}$ and if $\Hom(F',F)\rightarrow \Hom (F',M) \rightarrow 0$ is exact for every $F' \in \mathscr{F}$. If any morphism $f:F \rightarrow F$ such that $ \phi \circ f= \phi$ is an isomorphism, then it is called an $\mathscr{F}$-cover of $M$. If the class $\mathscr{F}$ is such that every object has an $\mathscr{F}$-cover, then $\mathscr{F}$ is called a precovering class. The dual notions are those of $\mathscr{F}$-envelope and enveloping class.
\end{definition}

\begin{definition}
A torsion theory  for a category $ \mathcal{C}$ is a pair $(\mathscr{T},\mathscr{F})$ of classes of objects $\mathcal{C}$ such that
\begin{enumerate}
\item $\Hom (T,F)=0 $ for all $T\in \mathscr{T}, F \in \mathscr{F}$.
\item If $\Hom (C,F)=0$ for all $F \in \mathscr{F}$, then $C \in \mathscr{T}$.
\item If $\Hom (T,C)=0$ for all $T \in \mathscr{T}$, then $C \in \mathscr{F}$.
\end{enumerate}
In that case, $\mathscr{T}$ is called a \emph{torsion class} and while $\mathscr{F}$ is called a \emph{torsion-free} class.
\end{definition}

In fact, being a torsion class is equivalent to being closed under quotient objects, coproducts and extensions. And its dual form is valid for a torsion-free class.

A torsion theory $(\mathscr{T},\mathscr{F})$ is called \emph{hereditary} if the torsion class is closed under subobjects, or equivalently, the torsion-free class is closed under injective envelopes. And  it is called of finite type when its torsion-free class is closed under direct limits. The prototypical example of a hereditary torsion theory of finite type comes from the category of modules over an integral domain where $\mathscr{T}$ is the class of all torsion modules and  $\mathscr{F}$ is the class of all torsion-free modules.

\section{Locally torsion-free quasi-coherent sheaves}
A scheme $(X,\os_X)$ is said to be {\it integral} if $\os_X(U)$ is an integral domain, for each open subset $U$ of $X$, or equivalently, if it is both reduced and irreducible scheme. Since the properties of being reduced and irreducible scheme are local, it may be reduced to the level of affine sets, that is,  $\os_X(U)$ is an integral domain, for each open affine subset $U\subset X$. From now on all schemes are assumed to be integral.

We start this section by proving that locally torsion quasi--coherent sheaves are easily shown to induce a torsion theory in $\qco(X)$:

\begin{proposition}\label{torsion}
Let $\mathscr{T}$ be the class of quasi-coherent sheaves over $X$ whose image on an affine open subset $U$ is torsion. Then $\mathscr{T}$ is a torsion class of a hereditary torsion theory in $\qco(X)$.
\end{proposition}
\begin{proof}
Since $\mathscr{T}$ is closed under extensions, quotients, coproducts and subobjects, it is a torsion part of the hereditary torsion theory $(\mathscr{T},\mathscr{F}_{\mathscr{T}})$ where $\mathscr{F}_{\mathscr{T}}$ consists of the $\mathcal{M}\in \qco(X)$ having just the zero morphism from each element of $\mathscr{T}$.
\end{proof}

Now let $\mathscr{F}$ be the class in $\qco(X)$ of {\it locally torsion-free quasi-coherent sheaves}, that is, $\mathcal{F}\in \mathscr{F}$ whenever $\mathcal{F}(U)$ is torsion-free $\mathcal{O}_X(U)$-module, for each affine open set $U$ in $X$. First of all, we claim that being  locally torsion-free is a Zariski-local notion in $\qco(X)$.
\begin{lemma}\label{local}
Let $R$ be an integral domain and $M$ be an $R$-module. Then the following are equivalent:
\begin{enumerate}
\item  $M$ is torsion-free.
\item  $M_P$ is torsion-free as $R_P$-module for each prime ideal $P$.
\item  $M_m$ is torsion-free as $R_m$-module for each maximal ideal $m$.
\end{enumerate}
\end{lemma}
\begin{proof}
 The implications $(1) \Rightarrow (2) \Rightarrow (3)$ are easy. For $(3)\Rightarrow (1)$, assume that there is  a nonzero torsion  element $x \in M$. Then the ideal $\Ann _R (x)=\{r\in R \mid \quad rx=0\}$ is neither zero nor $R$ since $x$ is not zero. Consider  the maximal ideal $m$ containing $\Ann_R(x)$. Then $\frac{m}{1}$ is not zero in $M_m$. But $\frac{r}{1}.\frac{m}{1}=0$ for any nonzero $r \in \Ann_R(x)$. By the assumption, $\frac{r}{1}=0$ in $R_m$, that is, $tr=0$ for some $t \in R \backslash m$. But since $r \neq 0$ and $R$ is an integral domain, $t=0$ which  yields to a contradiction.
\end{proof}

\begin{proposition}
Let $\mathcal{F}$ be a quasi-coherent sheaf over $X$. Then the following are equivalent:
\begin{enumerate}
\item $\mathcal{F} \in \mathscr{F}$.
\item There is a cover $\mathcal{U}$ of $X$ containing affine open subsets whose images under $\mathcal{F}$ are torsion-free.
\item $\mathcal{F}_x$ is torsion-free for each $x \in X$.
\end{enumerate}
\end{proposition}
\begin{proof}
$(1)\Rightarrow (2)$ is clear. Suppose $(2)$. Let $x\in X$. Since $\mathcal{U}$ is a cover of $X$ with affine open sets, there is an affine open set $U \in \mathcal{U}$ containing $x$. Since $\mathcal{F}$ is quasi-coherent, $\mathcal{F}_x=M_P$ where $\mathcal{F}(U)=M$ is torsion-free and for some prime ideal $P \in U$ corresponding to $x$. This proves $(3)$. Now assume $(3)$. Let $U$ be an affine open set. By assumption, $\mathcal{F}_x=(\mathcal{F}\mid _U)_x =\mathcal{F}(U)_P$ is torsion-free for all $x \in U$. Hence Lemma \ref{local} implies $(1)$.
\end{proof}

Since the collection of all affine open subsets of $X$ constitutes a base, we can define the  image of each open subset of $X$ as an inverse limit of affine open subsets of $X$ that are contained in the open subset considered. This means that being torsion-free on each open affine subset of $X$ implies being torsion-free on each open subset of $X$.

As for the torsion theory in Proposition \ref{torsion}, it is easy to see that $\mathscr{F}$ is contained in $\mathscr{F}_{\mathscr{T}}$ of the previous proposition. We may also consider in $\mathfrak{Sh}(X)$, the category of sheaves of $\mathcal{O}_X$-modules, the pair $(\overline{\mathscr{T}},\overline{\mathscr{F}})$ of locally torsion and locally torsion-free $\os_X$-modules. This is easily shown to be a torsion theory in $\mathfrak{Sh}(X)$. However it is not clear if this pair constitutes a torsion theory when we restrict it to $\qco(X)$. In pursuing this aim, we will focus on the elements of the class $\mathscr{F}$. It is not difficult to see that $\mathscr{F}$ is closed under subobjects and extensions. But proving that $\mathscr{F}$ is closed under products requires more work and will be the main goal of this section. We start with the following:

\begin{lemma}\label{monic}
If $\mathcal{F}$ is a quasi-coherent sheaf over $X$ which is in the class $\mathscr{F}$, then its restriction maps between affine open subsets are monomorphisms.
\end{lemma}
\begin{proof}
Let $V\subseteq U$ be affine open subsets and $\res_{UV}:
\mathcal{F}(U) \rightarrow \mathcal{F}(V)$ be its restriction map
between these affine open subsets. Suppose $0 \neq x \in
\mathcal{F}(U)$  such that $\res_{UV}(x)=0$. Then $ \id_ {\os_X(V)}
\otimes_{\os_X(U)} \res_{UV}(1 \otimes x)=0$. This implies that
$1\otimes x=0$ in $\os_X(V) \otimes_{\os_X (U)} \mathcal{F}(U)$.
Since $\os_X(V)$ is flat as $\os_X(U)$-module, as explained in
\citet[Proposition 8.8, I]{bo}, there exist a matrix $A_{n\times1}$
with coefficients from $\os_X(U)$ and a vector $S_{1 \times n}$
coefficients from $\os_X(V)$ such that $A.x=0$ and $1=S.A$. But
since $x$ is torsion-free and nonzero, $A=0$. This contradicts with
$1=S.A$.
\end{proof}

By Lemma \ref{monic} and the definition of sheaf, we can result that for every non-empty affine open subset $U\subseteq X$ and for every affine open covering $U=\bigcup_i U_i$, we have  $\mathcal{F}(U)=\bigcap_i\mathcal{F}(U_i)$.

In the next results we will analyze the interlacing between the product object in $\mathscr{F}$ in $\qco(X)$ and the product object in $\overline{\mathscr{F}}$ in $\mathfrak{Sh}(X)$.

\begin{proposition}\label{prodsub}
Let $\{\mathcal{F}_i\}_I$ be a family of torsion-free quasi-coherent sheaves over $X$. Its  direct product $\mathcal{F}$ in $\qco (X)$ is a subsheaf of $\prod_I \mathcal{F}_i$.
\end{proposition}
\begin{proof}
By definition  of the direct product, there is a unique morphism
$\alpha: \mathcal{F} \rightarrow \prod \mathcal{F}_i$. Then we need
to show that $\Ker \alpha$ is the zero sheaf. Note that $(\Ker
\alpha)(U)= \Ker \alpha_U$ for every open subset $U$. Firstly, we
will prove that $\Ker \alpha \in \qco(X)$. Consider the
following diagram with the one tensorized by $\os_X(V)$ for affine
open subsets $V \subseteq U$,
\begin{center}
{$\xymatrix{ \Ker \alpha_U \ar@{^{(}->}[d] \ar[rr]&& \Ker \alpha_V \ar@{^{(}->}[d]\\
 \mathcal{F}(U) \ar[rr]^-{\res_{UV}}\ar[d]_-{\alpha_U} && \mathcal{F}(V)\ar[d]^-{\alpha_V}\\
 \prod\mathcal{F}_i(U)\ar@{^{(}->}[rr]_{(\res ^i_{UV})_I}&&\prod \mathcal{F}_i(V)
 }$ \quad\quad $\xymatrix{ \os_X(V)\otimes_{\os_X(U)}\Ker \alpha_U \ar@{^{(}->}[d] \ar[rr]^-{f}&& \Ker \alpha_V \ar@{^{(}->}[d]\\
 \os_X(V)\otimes_{\os_X(U)}\mathcal{F}(U) \ar[rr]^-{g}\ar[d]_-{\id \otimes \alpha_U} && \mathcal{F}(V)\ar[d]_-{\alpha_V}\\
 \os_X(V)\otimes_{\os_X(U)}\prod\mathcal{F}_i(U)\ar@{^{(}->}[rr]_-{\id \otimes (\res ^i_{UV})_I}&&\prod \mathcal{F}_i(V) .}$}
 \end{center}
For short, $\iota_1:=\id_{\os_X(V) \otimes_ {\os _X(U)} \iota_{\Ker
\alpha_U}}$ and $\iota_2:=\iota_{\Ker \alpha_V}$ are inclusion
morphisms. Of course, the last morphism of the first diagram, $(\res
^i_{UV})_I: \prod\mathcal{F}_i(U) \rightarrow \prod
\mathcal{F}_i(V)$ is
monic by Lemma \ref{monic} and it is preserved in the second diagram
since $\os_X(V)$ is flat as $\os_X(U)$-module (where $\res^i_{UV}$ is the restriction map of
$\mathcal{F}_i$ between affine open subsets $V \subseteq U$). As a first
observation, it can be easily seen that  $ f:= \id_{\os_X(V)}
\otimes_{\os_X(U)} \res_{UV}\mid_{\Ker \alpha_U} : \os_X(V)
\otimes_{\os_X(U)} \Ker \alpha_U \rightarrow \Ker \alpha_V $, where
$\res_{UV}:\mathcal{F}(U) \rightarrow \mathcal{F}(V)$, is a
monomorphism  by the fact that $\mathcal{F}$ is quasi-coherent and
$\Ker \alpha_U \subseteq \mathcal{F}(U)$. Keeping in mind that the
morphism $g:=\id_{\os_X(V)} \otimes_{\os_X(U)} \res_{UV}$ is an
isomorphism, $\id \otimes \alpha_U \circ g^{-1}\circ \iota_2=0$ and
so by the universal property of the kernel, there is a morphism $f':
\Ker \alpha_V \rightarrow \os_X(V) \otimes_ {\os _X(U)} \Ker
\alpha_U $ such that $\iota_1 \circ f'= g^{-1} \circ \iota_2$. The
commutativity of the second diagram and the last equality of
morphisms helps us to get that $ f\circ f'=\iota_2 \circ f \circ
f'=g \circ \iota_1 \circ f'=g \circ g^{-1} \circ
\iota_2=\iota_2=\id_{\Ker \alpha_V}$ and $f' \circ f=i_1 \circ f'
\circ f =g^{-1} \circ \iota_2 \circ f=g^{-1} \circ g \circ
\iota_1=\iota_1=\id_{\os_X(V) \otimes_ {\os _X(U)} \Ker \alpha_U}$.
This means that $f$ is an isomorphism.

In fact, the morphism $\Ker \alpha \hookrightarrow \mathcal{F}
\rightarrow \mathcal{F}_i$ is the zero morphism for each $i \in I$.
The universality of the direct product $\mathcal{F}$ in $\qco(X)$
implies that $\Ker \alpha=0$.
\end{proof}
\begin{theorem}\label{prodtf}
The direct product $\mathcal{F}$ of a family $\{\mathcal{F}_i\}_{i\in I}$ of torsion-free quasi-coherent sheaves in $\qco(X)$ is the largest quasi-coherent subsheaf of $\prod_{i\in I} \mathcal{F}_i$. More concretely, it is  of the form
$$\mathcal{F}=\sum_{\substack{\mathcal{M}\in \qco(X)\\\mathcal{M} \subseteq \prod_{i\in I} \mathcal{F}_i}} \mathcal{M}.$$

\end{theorem}
\begin{proof}

By the Proposition \ref{prodsub} , we know that $\mathcal{F}$ is a quasi-coherent subsheaf of $\prod_{i\in I} \mathcal{F}_i$. Now, let  $\mathcal{F'}$ be a quasi-coherent subsheaf in $\prod_{i\in I} \mathcal{F}_i$. Consider the morphism
$$\xymatrix{\mathcal{F'} \ar@{^{(}->}[r] & \prod_{i\in I} \mathcal{F}_i \ar[r]^{\pi _i}&\mathcal{F}_i} $$
for each $i \in I$. By the universal property of the direct product, there is a unique morphism $f:\mathcal{F'} \rightarrow \mathcal{F}$ such that $(\pi_i \mid_{\mathcal{F'}})=(\pi_i \mid_{\mathcal{F}}) \circ f$. But for an open subset $U$ of $X$, the projection map $\pi_i (U) :\prod_{i\in I} \mathcal{F}_i(U) \rightarrow \mathcal{F}_i(U)$ is the canonical one. So we can deduce that the morphism $f$ that we have obtained is an inclusion. This proves that $\mathcal{F}$ is the largest quasi-coherent subsheaf of $\prod_{i\in I} \mathcal{F}_i$. It implies that
$$\mathcal{F}=\sum_{\substack{\mathcal{M}\in \qco(X)\\\mathcal{M} \subseteq \prod_{i\in I} \mathcal{F}_i}} \mathcal{M}.$$

\end{proof}

\begin{corollary}\label{tors}
The class $\mathscr{F}$ in $\qco(X)$ is closed under arbitrary direct products. In particular it induces a torsion theory of finite type in $\qco(X)$.
\end{corollary}
\begin{proof}

$\mathscr{F}$ is closed under direct limits and under arbitrary products in view of Proposition \ref{prodsub}. Since it is also closed under subobjects and extensions it is the right part of a torsion theory of finite type $(\mathscr{F'},\mathscr{F})$ in $\qco(X)$.
\end{proof}

On an integral scheme every flat quasi--coherent sheaf is locally torsion-free. Thus we immediately follow:
\begin{corollary}
The direct product $\mathcal{F}$ of a family $\{\mathcal{F}_i\}_{i\in I}$ of flat quasi-coherent sheaves in $\qco(X)$ is the largest quasi-coherent subsheaf of $\prod_{i\in I} \mathcal{F}_i$. More concretely, it is  of the form
$$\mathcal{F}=\sum_{\substack{\mathcal{M}\in \qco(X)\\\mathcal{M} \subseteq \prod_{i\in I} \mathcal{F}_i}} \mathcal{M}.$$
\end{corollary}

\begin{proof}
This follows by noticing that every flat quasi--coherent sheaf is in fact locally torsion-free.
\end{proof}

Now we get another proof of \cite[Proposition 4.16]{Saorin}.
\begin{corollary}
Let $X$ be a Dedekind scheme. The class ${\rm Flat}(X)$ of flat quasi--coherent sheaves is closed under taking products in $\qco(X)$.
\end{corollary}
\begin{proof}
Let $\{\mathcal{F}_i\}_{i\in I}$ be a family of flat quasi--coherent sheaves, hence a family of ${\mathcal O}_X$-modules in $\overline{\mathscr{F}}$. By Proposition \ref{prodsub}, the product object $\mathcal{F}$ in $\qco(X)$ is a subsheaf of $\prod_I \mathcal{F}_i$ which is locally torsion-free, because $\overline{\mathscr{F}}$ is a torsion-free class (so, in particular, closed under products). Hence $\mathcal{F}\in \mathscr{F}$. But for a Dedekind scheme the classes $\mathscr{F}$ and ${\rm Flat}(X)$ coincide, so we are done.
\end{proof}

\section{Torsion-free covers in $\qco(X)$}\label{tfcover}

In the previous Corollary \ref{tors} we showed that the class $\mathscr{F}$ of locally torsion-free quasi--coherent sheaves is the right part of a torsion theory in $\qco(X)$. One immediate consecuence of this is that each $\mathcal{M}\in \qco(X)$ admits an $\mathscr{F}$-\emph{reflection} and thus $\mathscr{F}$ is a \emph{reflective} class in $\qco(X)$ (see \cite{MacL} for notation and terminology and \cite{SRV} for a nice treatment and characterization of reflective subcategories). So, in particular, we deduce that $\mathscr{F}$ is enveloping.

This section is devoted to prove that the class $\mathscr{F}$ is also covering, that is, that each $\mathcal{M}\in \qco(X)$ admits an $\mathscr{F}$-cover.

Recall that a quasi-coherent sheaf $\mathcal{F}$ is said to be of \emph{type} $\kappa$, for $\kappa$ an infinite cardinal, if each $\mathcal{F}(U)$ is an $\os_X(U)$-module at most $\kappa$-generated for each affine open subset $U \subseteq X$. Let $\kappa$ be an infinite regular cardinal such that $\kappa > \mid \os_X (U)
\mid$ for each affine open subsets $U \subseteq X$ and  $\kappa > \mid
H\mid$, where $H:=\{\res_{UV} \mid \textrm{  for affine subsets }
V\subseteq U \subseteq X\}$.

On the other hand, given $\mathcal{F}\in \qco(X)$, we define the \emph{cardinality of }$\mathcal F$, $\mid \mathcal{F} \mid$, as $$\mid \mathcal{F} \mid={\rm sup}\{\mid \mathcal{F}(U)\mid:\ U\in \mathcal{U}\},$$ (here $\mathcal{U}$ stands for the set of all affine open subsets of $X$). Note that if $\kappa$ is as before, then $\mid \mathcal{F}\mid<\kappa$ if and only if $\mathcal{F}$ is of type $\kappa$.
\begin{lemma}\label{filt1}
Let $S$ be the set of isomorphism classes of quasi-coherent sheaves in $\mathscr{F}$ of type $\kappa$. Then $\mathscr{F}= \Filt(S)$.
\end{lemma}
\begin{proof}
Let $\mathcal{F} \in \mathscr{F}$ and  $x \in \mathcal{F}(U)$  for some affine open subset $U$. By \citet[Proposition 3.3]{relativehom}, there is a quasi-coherent pure subsheaf $\mathcal{G} \subseteq \mathcal{F}$ of $\kappa$ type  containing $x$. Here, the purity  is considered in the sense of tensor product. As known, it has the global and local properties on affine open subsets, i.e., $\mathcal{G}(U) $ is a pure submodule of $\mathcal{F}(U)$ for each affine open subset $U$. This implies that $r \mathcal{G}(U)= r\mathcal{F}(U)\cap \mathcal{G}(U)$ for all $r \in \os(U)$, and so $\mathcal{F}(U)/\mathcal{G}(U)$ is torsion-free for all affine open subset $U$, $\mathcal{F}/\mathcal{G} \in \mathscr{F}$.

By transfinite induction, we will construct an $S$-filtration for each object in $\mathscr{F}$. For $\mathcal{F} \in \mathscr{F}$, consider $\mu = \mid \mathcal{F} \mid$ and  $\mathcal{F}_0=0$, $\mathcal{F}_1:= \mathcal{G}$ obtained as above. For $\alpha < \mu$, if $x + \mathcal{G}_\alpha(U) \in (\mathcal{F}/ \mathcal{F}_\alpha)(U)=\mathcal{F}(U)/ \mathcal{F}_\alpha(U)$, there is a pure quasi-coherent  subsheaf $\mathcal{F}_{\alpha+1}/ \mathcal{F}_\alpha $ of type $\kappa$ containing $x + \mathcal{F}_\alpha$. For a limit ordinal $\beta \leq \mu $, $\mathcal{F}_\beta := \varinjlim _{\alpha < \beta} \mathcal{F}_\alpha$. Then, $(\mathcal{F}_\alpha \mid \textrm{ } \alpha \leq \mu)$ is an $S$-filtration for $\mathcal{F}$.

From that construction, we get that $\mathscr{F} \subseteq \Filt (S)$. Actually, $\Filt(S)= \mathscr{F}$. Indeed, if $(\mathcal{M}_\alpha \mid \textrm{ } \alpha \leq \lambda)$ is an $S$-filtration of a quasi-coherent sheaf $\mathcal{M}$, we have that $\mathcal{M}_1=\mathcal{M}_1/\mathcal{M}_0=\mathcal{M}_1/0$ is in $\mathscr{F}$. And now if we suppose that $\mathcal{M}_\alpha\in \mathscr{F}$, for $\alpha < \lambda$, we have a short exact sequence
$$0 \longrightarrow \mathcal{M}_\alpha \longrightarrow \mathcal{M}_{\alpha+1} \longrightarrow \mathcal{M}_{\alpha+1}/ \mathcal{M}_\alpha\longrightarrow 0,$$
where $\mathcal{M}_\alpha,\mathcal{M}_{\alpha+1}/ \mathcal{M}_\alpha $ are in $\mathscr{F}$. Therefore, $\mathcal{M}_{\alpha+1}$ is also in $\mathscr{F}$. Since $\mathscr{F}$ is closed under direct limits, $\mathcal{M}_\alpha$ is locally torsion-free whenever $\alpha$ is a limit ordinal. This implies that $\mathcal{M}_\lambda=\mathcal{M}$ is locally torsion-free.
\end{proof}

We will adapt the arguments of \cite{E} to the category $\qco(X)$ to infer in Theorem \ref{extfc} that $\mathscr{F}$ is covering. Since the set of affine open subsets of $X$ is a base of the scheme and uniquely determines quasi--coherent sheaves over it, we will often use the images of a quasi--coherent sheaf on affine  open subsets.
\begin{lemma}\label{l1}
Assume that $(\mathcal{F}_{\alpha})_{\alpha \leq \kappa}$ is a
filtration of $\mathcal{F}$ in $\qco(X)$. If $\mid \mathcal{F}' / \mathcal{ F} \mid <
\kappa$ where $\mathcal{F} \subseteq \mathcal {F} ' \in \qco(X)$,
then there is a filtration $(\mathcal{F}' _\alpha)_{\alpha \leq
\kappa}$ which is compatible with the one of $\mathcal {F}$ and
except for possibly one $\beta < \kappa$, $\mathcal{F}_{\alpha +1}/
\mathcal{F}_{\alpha}$ is isomorphic to $\mathcal{F}' _{\alpha +1}/
\mathcal{F}'_{\alpha}$ and $\mathcal{F}_{\beta +1}/
\mathcal{F}_{\beta}$ is a direct summand of $\mathcal{F}_{\beta +1}/
\mathcal{F}_{\beta}$ with the complement $\mathcal{F}' / \mathcal{F}$.
\end{lemma}
\begin{proof}
For each affine open $U \subseteq X$, $S_U$ denotes the $\os_X(U)$ submodule of
$\mathcal {F}'(U)$ which is generated by representatives of $\mathcal{F}'(U)/ \mathcal
{F}(U)$. By assumption, $\mid S_U \mid < \kappa $. We can complete
these subsets to a quasi-coherent subsheaf of $\mathcal{F}'$, say
$\mathcal{S}$, containing these submodules $S_U \subseteq \mathcal{S} (U)$ and with the
cardinality $< \kappa$.

We know that there exists $\beta_U < \kappa$ for each affine $U
\subseteq X$ such that $\mathcal {S} (U) \cap \mathcal
{F}_\beta=\mathcal {S} (U) \cap \mathcal {F}$ since $\kappa$ is a regular cardinal and is the length of the filtration. Consider  $\beta:=
\cup_U \beta_U$. Now define the new filtration as
$\mathcal{F}'_{\alpha}= \mathcal{F}_\alpha$ for $\alpha \leq \beta$,
and $\mathcal{F}'_{\alpha}= \mathcal{F}_\alpha + \mathcal{S}$ for
$\alpha > \beta$. Since these are quasi-coherent, $(\mathcal{F}'_{\alpha+1} /
\mathcal{F}'_{\alpha})(U)=\mathcal{F}'_{\alpha+1}(U) /
\mathcal{F}'_{\alpha}(U)$ and $(\mathcal{F}_\alpha +
\mathcal{S})(U)=\mathcal{F}(U) + \mathcal{S}(U)$ for affine open
subsets and by using the fact $\mathcal {S} (U) \cap \mathcal
{F}_\alpha(U) =\mathcal {S} (U) \cap \mathcal {F}(U)$ for each $\alpha
\geq \beta$, the  claims mentioned in the lemma follow.
\end{proof}
The next corollary says that in $\qco(X)$ it is possible  to convert a filtration of any length and whose quotient between consecutive factors is bounded by $\kappa$ into a filtration with $\kappa$-length. Recall that for a given class $\mathcal{C}$, $\Sum(\mathcal{C})$ is the class of direct sums of objects which are isomorphic to some in $\mathcal{C}$.
\begin{corollary}\label{filt2}
Let $\mathcal{C}$ be a class of quasi-coherent sheaves with
cardinality $< \kappa$. If a quasi-coherent sheaf $\mathcal {F}$ has
a $\mathcal{C}$-filtration, then it has a $\Sum
(\mathcal{C})$-filtration of length $\kappa$.
\end{corollary}
\begin{proof}
It easily follows by  making transfinite induction on the length of the given filtration and by using Lemma \ref{l1}.
\end{proof}
Let $\mathcal{F}' \in \Sum(\mathcal{C})$ with a given direct sum decomposition $\mathcal{F}'=\oplus_{i\in I} \mathcal{N}_i$ such that each $\mathcal{N}_i $ is isomorphic to some object in the class $\mathcal{C}$, 
As defined in \cite{E} for modules, we call a quasi--coherent subsheaf $\mathcal{F}\subseteq \mathcal{F}'$,to be  a nice  subsheaf relative to this direct sum decomposition if $\mathcal{F}=\oplus_{j\in J} \mathcal{N}_j$ for some subset $J \subseteq I$. And $\mathcal{F}$ is a nice subsheaf of $\mathcal{F}' \in \Filt(\Sum(\mathcal{C}))$ if, when we give $\mathcal{F}$ the induced filtration $(\mathcal{F}_\alpha)_{\alpha \leq \sigma}$, the image of the canonical map   $\mathcal{F}_{\alpha+1}/\mathcal{F}_\alpha \rightarrow \mathcal{F}'_{\alpha+1}/\mathcal{F}'_\alpha$ is a nice subsheaf of $\mathcal{F}'_{\alpha+1}/\mathcal{F}'_\alpha$ relative to the given direct sum decomposition of $\mathcal{F}'_{\alpha+1}/\mathcal{F}'_\alpha$ for each $\alpha < \sigma$.
\begin{lemma}\label{l2}
Let $\mathcal{F}$ be a quasi--coherent sheaf and $\mathcal{M} \in \Sum (\mathcal{F})$. Assume that we have a morphism $f: \mathcal {M} \rightarrow \mathcal{N}$ in $\qco(X)$. Then there exist a nice quasi-coherent subsheaf   $\mathcal{T}$ contained in $\Ker (f)$ relative to a direct sum decomposition of $\mathcal{M} \in \Sum(\mathcal{F})$   such that  $\mid \mathcal{M}/ \mathcal{T} \mid \leq \mid \mathcal{F} \mid ^{\mid \Hom(\mathcal{F},\mathcal{N})\mid}$.
\end{lemma}
\begin{proof}
Note that $\mathcal{M} = \oplus_{i \in I} \mathcal{F}_i$, where $\mathcal{F}_i =\mathcal{F}$ for each $i \in I$. Then  the morphism $f$ is of the form $(f_i)_{i\in I}$, where $f_i :\mathcal{F}_i \rightarrow \mathcal{N}$ for each $i \in I$. Now, we define an equivalence relation on $I$  for the fixed morphism $f=(f_i)_{i\in I}$  as follows: $i \sim j$ if and only if $f_i=f_j$. If $J$ is the subset of $I$ which represents the equivalence classes, we define the subpresheaf $\mathcal {F}'_i$ of $\mathcal{M}$ for each $i \in I \backslash J $ with $j \in J$ and $i \sim j$  such that $\mathcal{F}'_i(U)$ includes datas, on each open subset $U \subseteq X$,
  having  $x\in \mathcal{F}(U)$ in the $i$'th component, $-x$ in the $j$'th one and $0$ in the others. This subpresheaf is isomorphic to $\mathcal{F}_i(U)$. So it is a quasi-coherent sheaf which is isomorphic to $\mathcal{F}_i$. And also $f_U(\mathcal{F}'_i)=0$ for each $i\in I\backslash J$ and for each  affine open subset $U \subseteq X$. Then $\mathcal{T}:=\oplus_{i \in I \backslash J} \mathcal{F}'_i \subseteq \Ker f$.  The exact sequence
 $$\xymatrix{0 \ar[r] & (\oplus_{i \in I\backslash J} \mathcal{F}'_i)(U)\ar@{^{(}->}[r] &(\mathcal{M})(U) = (\oplus_{i \in I} \mathcal{F}_i)(U) \ar[r]^-{h_U} & (\oplus_{i \in J} \mathcal{F}_i)\ar[r]&0},$$
 where the map $h_U((x_i)_{i\in I})=(y_j)_{j\in J}$, $y_j=\sum_{\substack{i \in I\\i \sim j }} x_i$, is splitting since the map $t$ which is defined on each affine subset $U\subseteq X$ as $t_U: \mathcal{M}(U) \rightarrow (\oplus_{i \in I\backslash J} \mathcal{F}'_i)(U)$, $t_U(x_i)_{i \in I}=(y_i)_{i\in I}$ such that $y_i=x_i$ for each $i\in I$ and $y_j= \sum _{\substack{i\in I \backslash J\\i \sim j}}x_i$ for each $j \in J$ is compatible with restriction maps and gives us the identity map when composed with the inclusion map from $(\oplus_{i \in I\backslash J} \mathcal{F}'_i)(U)$. That means, $\mathcal{M} \cong (\oplus_{i \in J} \mathcal{F}_i) \oplus (\oplus_{i \in I\backslash J} \mathcal{F}'_i)$. So, the quasi--coherent subsheaf $\mathcal{T}=\oplus_{i \in I\backslash J} \mathcal{F}'_i$ is nice in $\mathcal{M}$ relative to some direct sum decomposition in $\Sum(\mathcal{F})$. Finally, $\mid J \mid \leq \mid \Hom(\mathcal{F},\mathcal{N})\mid$ implies that $\mid \oplus_{i \in J} \mathcal{F}_i \mid=\mid \mathcal{M}/ \mathcal{T} \mid \leq \mid \mathcal{F} \mid ^{\mid \Hom(\mathcal{F},\mathcal{N})\mid}$.
 \end{proof}
\begin{lemma}\label{filt3}
If $\mathcal{M} \in \Sum(\mathcal{F})$ and $\varepsilon:= 0 \rightarrow \mathcal{N} \rightarrow \mathcal{T} \rightarrow \mathcal{M}\rightarrow 0$ is an exact sequence in $\qco(X)$, then it is isomorphic to an exact sequence
$0 \rightarrow \mathcal{N} \rightarrow \mathcal{T}' \oplus \mathcal{V} \rightarrow \mathcal{M}' \oplus \mathcal{V}\rightarrow 0$ where $\mathcal{V}$ is a nice subsheaf of $\mathcal{M}$ and
$\mid \mathcal{M}/\mathcal{V}\mid \leq \mid \mathcal{F} \mid ^{\mid \Ext(\mathcal{F},\mathcal{N})\mid}$
\end{lemma}

\begin{proof}
Suppose $\mathcal{M}=\oplus_{i \in I} \mathcal{F}_i$ where $\mathcal{F}_i=\mathcal{F}$ for all $i \in I$. Then for each $i\in I$, we can consider the quasi--coherent sheaf associated to the one defined on affine open subsets $U \subseteq X$ as $f^{-1}(\mathcal{F}_i)(U)=f^{-1}_U (\mathcal{F}_i(U))$ with the map $f_U :  f^{-1}_U (\mathcal{F}_i(U)) \rightarrow \mathcal{F}_i$. Since it satisfies the condition of quasi--coherence on affine open subsets and commutativity on affine open subset inclusion, it is possible to find such a quasi--coherent sheaf and a commutative diagram
$$\xymatrix{0\ar[r] & \mathcal{N}\ar[r]\ar[d] & f^{-1}(\mathcal{F}_i)\ar[r]\ar@{^{(}->}[d] &\mathcal{F}_i \ar[r]\ar@{^{(}->}[d] &0\\
 0\ar[r] & \mathcal{N}\ar[r] & \mathcal{T}\ar[r] &\oplus_{i\in I}\mathcal{F}_i \ar[r] &0}.$$
Now, as done before, we define a equivalence relation on $I$ as $i \sim j$ if and only if there exists a commutative diagram morphism
$$\xymatrix{0\ar[r] & \mathcal{N}\ar[r]\ar[d]_{\id} & f^{-1}(\mathcal{F}_i)\ar[r]\ar[d]_{h_{ij}} &\mathcal{F}_i=\mathcal{F} \ar[r]\ar[d]_{\id} &0\\
 0\ar[r] & \mathcal{N}\ar[r] & f^{-1}(\mathcal{F}_i)\ar[r] &\mathcal{F}_j=\mathcal{F} \ar[r] &0}.$$
Consider the set $J$ of representatives of equivalence classes. Then, for each $i \in I \backslash J$ we  define a quasi--coherent subsheaf $\mathcal{V}_i$ of $\mathcal{M}$ with  properties on affine open subsets $U \subseteq X$ that  $\mathcal{V}_i(U)$ consist of elements from $\mathcal{M}=\oplus_I \mathcal{F}_i$ having $c_U$ in $i$'th and $-c_U$ in $j$'th component, where $c_u \in \mathcal{F}(U)$  and $0$ for others and where $j \in J$ with $j \sim i$. It is easy to see that $\mathcal{V}_i$ is isomorphic to $\mathcal{F}_i$. In fact, for a fixed $i_0 \in I \backslash J$, if we define a map $t_U: (\oplus_I \mathcal{F}_i)(U)=\oplus_I \mathcal{F}_i(U) \rightarrow \mathcal{V}_{i_0}(U)$ such that $t_U((c_U^i)_{i \in I})=(c'^i_U)_{i \in I}$ where $c'^{i_0}_U=c^{i_0}_U$, $c'^j_U=-c'^{i_0}_U$, where $j\in J$ and $j \sim i_0$, and $c'^i_U=0$. Then, $t_u \circ i_u=\id$ for each affine set $U$. This map is compatible with the restriction map on all affine open subsets $V \subseteq U \subseteq X$. So it can be extended to $\oplus_I\mathcal{F}_i$. Finally, $\mathcal{V}_{i_0}$ is a direct summand of $\oplus_I\mathcal{C}_i$, for each $i_0\in I$.

Now, we define a quasi--coherent sheaf $f^{-1}(\mathcal{V}_i)$ for each $i \in I \backslash J$ with $j \sim i$ and $j \in J$ such that $$(f^{-1}(\mathcal{V}_i))(U):=$$ $$f_U^{-1}(\mathcal{V}_i)=\{a_u + b_u \mid \textrm{ } a_u \in f^{-1}_U(\mathcal{F}_i(U)), b_u \in f^{-1}_U(\mathcal{F}_j(U)) \textrm{ and } -\pi_i f_U(a_u)=\pi_j f_U (b_u)  \}.$$
We have a commutative diagram with exact rows for each affine $U$ and $i \in I \backslash J$
$$\xymatrix{0\ar[r] & \mathcal{N}(U)\ar[r]^-{\iota_u}\ar[d] & (f^{-1}(\mathcal{V}_i))(U)\ar[r]^-f\ar@{^{(}->}[d] &\mathcal{V}_i(U) \ar[r]\ar@{^{(}->}[d] &0\\
 0\ar[r] & \mathcal{N}(U)\ar[r] & \mathcal{T}(U)\ar[r] &\oplus_{i\in I}\mathcal{F}_i(U) \ar[r] &0}.$$
With the morphism on affine set $U$ given by $\sigma_U:(f^{-1}(\mathcal{V}_i))(U) \rightarrow \mathcal{N}(U)$, $\sigma_U(a_U+b_U):=(h_{ij})_U(a_U)+b_U$, which  is compatible with restriction maps, we have $\sigma_U \circ \iota_U=\id_{\mathcal{N}(U)}$, that is, the first row is splitting for each affine set $U$. Then  $\mathcal{V}_i$ has an isomorphic image which is a direct summand in $f^{-1}(V_i)$. Since $\mathcal{V}_i$ is also a direct summand in $\mathcal{M}$, we can deduce that $\mathcal{V}_i$ has an isomorphic image in $\mathcal{T}$, which is a direct summand. Combining all of them and considering the quasi--coherent subsheaf $\mathcal{V}:=\oplus_{i \in I \backslash J}\mathcal(V)_i$ of $\mathcal{M}$, we identify it with its isomorphic image in $\mathcal{T}$. So, the original exact sequence is reduced to the desired one. And also $\mathcal{M} / \mathcal{V} \simeq \oplus_{j \in J}\mathcal{F}_j$, the claim on the cardinality follows.

\end{proof}
Returning to our case of the class of locally  torsion-free quasi-coherent sheaves, we can combine all previous results to infer the following
\begin{lemma}\label{filt4}
Let $\lambda$ be a cardinal. There is a cardinal $\mu$ such that for each morphism $f: \mathcal{M}\rightarrow \mathcal{N}$ where $\mathcal{M} \in \Filt (S)$ and $\mid \mathcal{N}\mid \leq \lambda$ there is a quasi--coherent subsheaf $\mathcal{T}$ of $\mathcal{M}$ contained in $\Ker f$ such that $\mathcal{F}/\mathcal{T} \in  \Filt(\mathcal{S})$ and $\mid \mathcal{F}/ \mathcal{T}\mid \leq \mu$.
\end{lemma}
\begin{proof}
Using Corollary  \ref{filt2}, Lemma \ref{l2} and \ref{filt3}, we can apply a transfinite induction on $\kappa$ mentioned in Corollary \ref{filt2} to find a cardinal as done in \citet[Theorem 5.1]{E}.
\end{proof}

\begin{theorem}\label{extfc}
Each quasi-coherent sheaf over $X$ has an $\mathscr{F}$-cover.
\end{theorem}

\begin{proof}
Let $\mathcal{N}$ be a quasi-coherent sheaf and $\lambda$ be the cardinality of $\mathcal{N}$. By Lemmas \ref{filt1} and \ref{filt4}, there is a cardinal $\mu$ such that, for each morphism $f: \mathcal{F} \rightarrow \mathcal{N}$ where $\mathcal{F} \in \mathscr{F}$, there is  a pure submodule $\mathcal{T}$ of $\mathcal{F}$ contained in $\Ker f$ with $\mathcal{F}/\mathcal{T} \in \mathscr{F}$ of type $\mu$. Then $\mathcal{M}:=\oplus_{\mathcal{T} \in S' \\ h:\mathcal{T}\rightarrow \mathcal{N}} \mathcal{T}_h$, with the canonical morphism $\sigma : \mathcal{M} \rightarrow \mathcal{N} $, is a precover of $\mathcal{N}$ (where $S'$ is the isomorphism class of locally torsion-free quasi-coherent sheaves of type $\mu$). Finally, since $\mathscr{F}$ is closed under direct limits, by \cite[Theorem 2.2.12]{Xu} (whose proof is valid
for any Grothendieck category, so in particular for $\qco(X)$) $\mathcal{M}$ has an $\mathscr{F}$-cover.
\end{proof}

\bigskip\par\noindent {\bf Remark.} Theorem \ref{extfc} may be also derived from Lemma \ref{filt1} and \cite[Corollary 2.15]{Saorin}.


\begin{thebibliography}{00}
\bibliographystyle{apalike}
\addcontentsline{toc}{chapter}{REFERENCES}
\thispagestyle{myheadings}


\bibitem[Ad\'{a}mek \& Rosick\'{y}, 1994]{adamek} Ad\'{a}mek, J., \& Rosick\'{y}, J. (1996). Locally presentable and accessible categories. \emph {London Mathematical Society Lecture Note Series,  189}. Cambridge: Cambridge University Press.



\bibitem[Bass, 1960]{bass} Bass, H. (1960). Finitistic dimension and a homological generalization of semiprimary rings.  \emph{Trans. Amer. Math. Soc., 95}, 466-488.


\bibitem[Bueso et al., 1991]{blas} Bueso, J.L., Torrecillas, B. \& Verschoren A. (1991).
Generalized Local Cohomology and
Quasicoherent Sheaves. \emph{J. Algebra, 138}, 298-312.

\bibitem[Conrad, 2000]{conrad} Conrad, B. (2000). Grothendieck Duality and Base Change. Berlin: Springer, Lecture Notes in Mathematics, vol. 1750.


\bibitem[Crawley-Boevey, 1994]{cb} Crawley-Boevey, W. (1994). Locally finitely presented additive categories. \emph{ Comm. Algebra, 22}, 1641–1674.

\bibitem[Dickson, 1966]{dickson} Dickson, S.E. (1966). A torsion theory for abelian categories. \emph{Trans. Amer. Math. Soc., 121}, 223-235.

\bibitem[Enochs, 2012]{E} Enochs, E. (2012) Shortening filtrations. \emph{Science China Math., 55}, 687-693.


\bibitem[Enochs, 1963]{torsionfree} Enochs, E. (1963). Torsion-free covering. \emph{Proc. Amer. Math. Soc., 14}, 884-889.
\bibitem[Enochs \& Estrada, 2005]{relativehom} Enochs, E., \& Estrada, S. (2005). Relative homological algebra in the category of quasi-coherent sheaves. \emph{Advances in Mathematics, 194}, 284-295.


\bibitem[Estrada et al., 2012]{EGPT} Estrada, S., Guil Asensio, P. A., Prest, M. \& , Trlifaj J. (2012). Model category structures arising from Drinfeld vector bundles. \emph{Advances Math., 231}, 1417-1438.

\bibitem[Estrada \& Saorin, 2013]{ES} Estrada, S., Saor\'{\i}n, M. Locally finitely presented categories with no flat objects. \emph{Forum. Math.}, to appear. Available at arXiv:1204.5681.

\bibitem[Gillespie, 2007]{Gillespie}Gillespie, J. (2007). Kaplansky classes and derived categories. \emph{ Math. Z., 257}, 811–843.

\bibitem[Golan \& Teply, 1973]{teply}Golan, J.S., \& Teply, M.L. (1973). Torsion-free covers. \emph{Israel J. Math., 15}, 237-256.

\bibitem[Grothendieck, 1957]{Grothendieck} Grothendieck, A. (1957). Sur quelques points d\'alg\`ebre homologique.\emph{ T\^ohoku Math. J. 9}, 119–221.

\bibitem[Hosseini \& Salarian, 2012]{HosSal} Hosseini, E., \& Salarian, S. (2012). A cotorsion theory in the homotopy category of flat quasi-coherent sheaves. \emph{Proc. Amer. Math. Soc.,} \url{http://dx.doi.org/10.1090/S0002-9939-2012-11364-4}.

\bibitem[Hovey, 2001]{Hovey}Hovey, M. (2001). Model category structures on chain complexes of sheaves, \emph{ Trans. Amer. Math. Soc. 353}, 2441–2457.

\bibitem[MacLane, 1971]{MacL}MacLane, S. (1971). Categories for the Working Mathematician, \emph{Springer-Verlag}.


\bibitem[Murfet, 2007]{Murfet} Murfet, D. The Mock homotopy category of
projectives and Grothendieck duality, PhD thesis, Australian National
University, 2007, available at
\url{http://www.therisingsea.org/thesis.pdf}.

\bibitem[Murfet \& Salarian, 2011]{MurSal} Murfet, D.,\&  Salarian, S. (2011). Totally acyclic complexes over noetherian schemes, \emph{Advances
Math., 296}, 1096–1133.

\bibitem[Saor\'{\i}n et al., 2000]{SRV} Rada, J.,\& Saor\'{\i}n, M., del Valle, A. (2000), Reflective subcategories, \emph{Glasgow Math. J. 42}, 97-113. 

\bibitem[Saor\'{\i}n  \& \v{S}\v{t}ov\'{\i}\v{c}ek, 2011]{Saorin} Saor\'{\i}n, M., \& \v{S}\v{t}ov\'{\i}\v{c}ek, J. (2011). On exact categories and applications to triangulated categories. \emph{Advances Math., 228}, 968-1007.

\bibitem[Stenstrom, 1975]{bo} Stenstrom, B. (1975). Ring of quotients. Berlin Heidelberg New York: Springer-Verlag.

\bibitem[Xu, 1996]{Xu} Xu, J. (1996). Flat covers of modules. \emph{Lecture Notes in Mathematics, 1634}. Berlin:
Springer-Verlag.

\end{thebibliography}
\end{document}